\documentclass{article}

\usepackage{amsmath,amsfonts,amssymb,amsthm,color, latexsym, setspace}
\usepackage[colorlinks,citecolor=red,urlcolor=blue,hypertexnames=false]{hyperref}

\newtheorem{theorem}{Theorem}
\newtheorem{lemma}[theorem]{Lemma}

\newtheorem{definition}[theorem]{Definition}
\newtheorem{corollary}[theorem]{Corollary}
\newtheorem{example}[theorem]{Example}

\newcommand\RR{\mathbb{R}}
\newcommand\CC{\mathbb{C}}
\newcommand\ZZ{\mathbb{Z}}

\begin{document}

\allowdisplaybreaks[4]

\author{D. Barbieri\footnote{Universidad Aut\'onoma de Madrid, 28049 Madrid. E-mail\,$:$ davide.barbieri@uam.es}, E. Hern\'andez\footnote{Universidad Aut\'onoma de Madrid, 28049 Madrid. E-mail\,$:$ eugenio.hernandez@uam.es}\, and A. Mayeli\footnote{City University of New York, Queensborough and the Graduate Center. E-mail\,$:$ amayeli@gc.cuny.edu}}

\title{Lattice sub-tilings and frames in LCA groups}

\maketitle

\begin{abstract}
Given a lattice $\Lambda$ in a locally compact abelian group $G$ and a measurable subset $\Omega$ with finite and positive measure, then the set of characters associated to the dual lattice form a frame for $L^2(\Omega)$ if and only if the distinct translates by $\Lambda$ of $\Omega$ have almost empty intersections. Some consequences of this results are the well-known Fuglede theorem for lattices, as well as a simple characterization for frames of modulates.\\
\emph{Keywords}: Tiling sets, frames of exponentials, systems of translates.\\
\emph{MSC 2010}: 43A25, 42C15, 52C22
\end{abstract}

\section{Introduction}

Let $G$ denote a locally compact and second countable abelian group (LCA group). A closed subgroup $\Lambda$  of $G$  is called a \emph{lattice} if it is  discrete and  co-compact, i.e, the quotient group $G/\Lambda$ is compact. Recall that, since $G$ is second countable, then any discrete subgroup of $G$ is also countable (see e.g. \cite[Section 12, Example 17]{Pontryagin}). Assume that $G$ is abelian, and denote the dual group by $\widehat G$. The dual lattice of $\Lambda$ is defined as follows:
\begin{align}
\Lambda^\perp=\{ \chi\in \widehat G: \  \langle \chi,\lambda\rangle=1
\ \forall \lambda\in \Lambda\},
\end{align}
where $\langle \chi, \lambda\rangle$ indicates the action of  character $\chi$ on the group element $\lambda$. 

We recall that, by the duality theorem between subgroups and quotient groups (see e.g. \cite[Lemma 2.1.2]{Rudin-FA}), the dual lattice $\Lambda^\perp$ is a subgroup of $\widehat G$ that is topologically isomorphic to the dual group of $G/\Lambda$, i.e., $\Lambda^\perp \cong
\widehat{G/\Lambda}$. Moreover, since $G/\Lambda$ is compact, the dual lattice $\Lambda^\perp$ is discrete. Notice also that $\widehat G/\Lambda^\perp \cong \widehat\Lambda$, which implies that $\Lambda^\perp$ is co-compact, hence it is a lattice. \\

Let $dg$ denote a Haar measure on $G$. For a function $f$ in $L^1(G)$, the Fourier transform of $f$ is defined by
$$
\mathcal F_G(f)(\chi) = \int_G f(g) \overline {\langle \chi, g\rangle}\, dg\,, \qquad \chi\in \widehat G,
$$
where $\langle \chi, g\rangle$ denotes the action of the character $\chi$ on $g.$ By the inversion theorem \cite[Section 1.5.1]{Rudin-FA}, a Haar measure $d\chi$ can be chosen on $\widehat G$ so that the Fourier transform $\mathcal F_G$ is an isometry from $L^2(G)$ onto $L^2(\widehat G).$ More precisely, 
\begin{align} \label{Plancherel}
\langle \mathcal F_G(f) , \mathcal F_G(g) \rangle_{L^2(\widehat G,d\chi)} =
\langle f , g \rangle_{L^2(G, dg)} \qquad f, g \in L^2(G)\,.
\end{align}

For any $\chi\in \widehat G$, we define the {\it exponential function} $e_{\chi}$ by
$$
e_{\chi}: G\to \CC, \quad \quad e_{\chi}(g):= \langle \chi, g\rangle.
$$
For any measurable subset $\Omega$ of $G$, we let $|\Omega|$ denote the Haar measure of $\Omega$. Throughout this paper, we let ${\bf 1}_\Omega$ denote the characteristic function of the set $\Omega$. We shall also use the addition symbol '$+$' for the group action, and $0$ for the neutral element, since $G$ is abelian.

\begin{definition}[Sub-Tiling]
Let $\Omega\subset G$ be a measurable set with finite and positive Haar measure, and let $\Lambda$ be a lattice subgroup of $G$. We say that $(\Omega, \Lambda)$ is a {\it sub-tiling pair} for $G$ if
\begin{align}\label{tiling-property}
\sum_{\lambda\in\Lambda} {\bf 1}_\Omega(g-\lambda) \leq 1 \quad a.e. \
g\in  G\,.
\end{align}
\end{definition}
By replacing the inequality with an equality, the definition is that of a \emph{tiling} pair. In this weaker form, it is equivalent to say that the translates of $\Omega$ by elements of $\Lambda$ are a.e. disjoint, i.e. $(\Omega, \Lambda)$ is a sub-tiling pair for $G$ if and only if
$$
|\Omega \cap (\Omega + \lambda)| = 0 \quad \forall \ \lambda \in \Lambda \, , \ \lambda \neq 0 .
$$
Observe also that any sub-tiling set is a subset of a tiling set.

\

Our main result is the following. Recall that a cross section $Q_\Lambda \subset G$ for a group $G$ and a lattice $\Lambda$ is a measurable set of representatives of $G/\Lambda$.

\begin{theorem}[Main Result]\label{main} Let $\Lambda$ be a lattice in $G$, let $\Omega \subset G$ be a set with finite and positive measure, and let $Q_\Lambda \subset G$ be a cross section for $G/\Lambda$. Then the following are equivalent.
\begin{itemize}
\item[1)] The pair $(\Omega,\Lambda)$ is sub-tiling for $G$.
\item[2)] For a.e. $\chi\in \widehat G$ it holds
$$
\sum_{\tilde\lambda\in \Lambda^\perp} |\mathcal F_G({\bf 1}_\Omega) (\chi+\tilde\lambda)|^2 = |Q_\Lambda|\,|\Omega| .
$$
\item[3)] The system of translates
$\{\sqrt{|\Omega|}^{~-1}{\bf 1}_\Omega (\cdot-\lambda): \ \lambda\in \Lambda\}$ is orthonormal in $L^2(G)$.
\item[4)] The exponential set $E_{\Lambda^\perp} = \{e_{\tilde\lambda}: \tilde\lambda\in\Lambda^\perp\}$ is a frame for $L^2(\Omega)$.
\end{itemize}
Moreover, if any of the above conditions holds, then the frame in point 4) is tight, with constant $|Q_\Lambda|$.
\end{theorem}

As a first corollary we can obtain the following result, which was proved by B. Fuglede in the Euclidean setting \cite{Fue74}, and in the present setting by S. Pedersen with a different approach (\cite{Pedersen}).
\begin{corollary}\label{cor:Pedersen}
A set of finite and positive measure $\Omega$ tiles $G$ with translations by $\Lambda$ if and only if the exponential set $E_{\Lambda^\perp}$ is an orthogonal basis for $L^2(\Omega)$.
\end{corollary}

Let us now denote with $M : \Lambda^\perp \to \mathcal{U}(L^2(\Omega))$  the modulations $M_{\tilde\lambda}f(g) = e_{\tilde\lambda}(g)f(g)$. As a second consequence of Theorem \ref{main} we obtain the following.
\begin{corollary}\label{cor:modulates}
Conditions 1) - 4) of Theorem \ref{main} are equivalent to\vspace{4pt}\\
\phantom{\hspace{10pt}} 5) The system of modulates $\Psi_{\Lambda^\perp} = \{M_{\tilde\lambda}\psi : \tilde\lambda\in\Lambda^\perp\}$ is a frame for $L^2(\Omega)$, with frame bounds $0 < A|Q_\Lambda| \leq B |Q_\Lambda| < \infty$, for any $\psi \in L^2(\Omega)$ satisfying $$0 < A \leq \textnormal{ess}\inf |\psi|^2 \leq \textnormal{ess}\sup |\psi|^2 \leq B < \infty.$$
\end{corollary}

The novelty of this paper is that relates subtilings with frames of exponentials. Moreover, the proof of 2)\,$\Rightarrow$\,3) in Theorem \ref{main} is shown by using the bracket map of a system of translates (Corollary \ref{orthonormal translation system}) that have been introduced to study properties of translation invariant spaces (see \cite{HSWW} and the references therein.)

The setting of LCA groups allows to prove simultaneously results for a large variety of groups, namely $\mathbb R^n, \mathbb Z^d, \mathbb T^k$, and all finite groups $F$ with discrete topology, as well as the so called elementary LCA groups $G = \mathbb R^n\times \mathbb Z^d \times \mathbb T^k \times F$. Observe that knowing Theorem \ref{main} and Corollary \ref{cor:Pedersen} for every factor group does not immediately provide the corresponding results for $G$.

\

The motivation for this paper comes from the problem of studying the relationship between spectrum sets and tiling pairs, whose roots dates back to a 1974 paper of B. Fuglede (\cite{Fue74}). There he proved that a set $E \subset \RR^d$, $d\geq 1$, of positive Lebesgue measure, tiles $\RR^d$ by translations with a lattice $\Lambda$ if and only if $L^2(E)$ has an orthogonal basis of exponentials indexed by the annihilator of $\Lambda.$  A more general statement in $\RR^d$, which says that if $E \subset \RR^d$, $d\geq 1$, has positive Lebesgue measure, then $L^2(E)$ has an orthogonal basis of exponentials (not necessary indexed by a lattice) if and only if $E$ tiles $\RR^d$ by translations, has been known as the Fuglede Conjecture.

A variety of results were proved establishing connections between tiling and orthogonal exponential bases. See, for example, \cite{LRW00}, \cite{IP98}, \cite{L02}, \cite{KL03} and \cite{KL04}. In 2001, I. Laba proved that the Fuglede conjecture is true for the union of two intervals in the plane (\cite{L01}). In 2003, A. Iosevich, N. Katz and T. Tao (\cite{IKT03}) proved that the Fuglede conjecture holds for convex planar domains. It was also proved in \cite{IP98} and \cite{LRW00} that $\Lambda$ tiles $\mathbb{R}^d$ by the unit cube $Q^d$ if and only if $\Lambda$ is a spectral set for $Q^d$. In 2004, T. Tao (\cite{T04}) disproved the Fuglede Conjecture in dimension $d=5$ and larger, by exhibiting a spectral set in $\RR^{5}$ which does not tile the space by translations. In \cite{KM06}, M. Kolountzakis and M. Matolcsi also disproved the reverse implication of the Fuglede Conjecture for dimensions $d=4$ and higher. In \cite{FR06} and \cite{FMM06}, the dimension of counter-examples was further reduced. In fact, B. Farkas, M. Matolcsi and P. Mora show in \cite{FMM06} that the Fuglede conjecture is false in $\RR^3$. The general feeling in the field is that sooner or later the counter-examples of both implications will cover all dimensions. However, in \cite{IMP15} the authors  showed that the Fuglede Conjecture holds in two-dimensional vector spaces over prime fields. Then, in \cite{AI} the authors prove that tiling implies spectral in $\mathbb{Z}_p^3$, $p$ prime, and Fuglede conjecture is true for $\mathbb{Z}_2^3$ and $\mathbb{Z}_3^3$. Very recently, important developments in LCA groups, with crucial implications on sampling theory, have been developed by E. Agora, J. Antezana and C. Cabrelli in \cite{AAC}, where the authors could obtain a full characterization of Riesz bases for multi-tiling sets in LCA groups.

\vspace{1ex}
\noindent
{\bf Acknowledgements}: D. Barbieri was supported by a Marie Curie Intra European Fellowship (626055) within the 7th European Community Framework Programme. D. Barbieri and E. Hern\'andez were supported by Grant MTM2013-40945-P (Ministerio de Econom\'ia y Competitividad, Spain). A. Mayeli was supported by PSC-CUNY-TRADB-45-446, and by the Postgraduate Program of Excellence in Mathematics at Universidad Aut\'onoma de Madrid from June 19 to July 17, 2014, when this paper was started. The authors wish to thank Alex Iosevich for several interesting conversations regarding this paper and his expository paper on the Fuglede conjecture for lattices \cite{Iosevich-Expository}.

\section{Notations and Preliminaries} \label{Pre}

Let  $\Lambda$ be a lattice in an LCA group $G$. Denote by $Q_\Lambda \subset G$ a measurable cross section of $G/\Lambda.$  By definition, a cross section is a set of representatives of all cosets in $G/\Lambda $, so that the intersection of $Q_\Lambda$ with any coset $g+\Lambda$ has only one element. The existence of a Borel measurable cross section is guaranteed by \cite[Theorem 1]{FG68}.
Moreover, it is evident that ($Q_\Lambda$, $\Lambda$) is a tiling pair for $G$, while any tiling set $\Omega$ differs from a cross section at most for a zero measure set. \\

Let $d\dot g$ be a normalized Haar measure for $G/\Lambda$. Then the relation between Haar measure on $G$ and Haar measure for $G/\Lambda$ is given by {\it Weil\rq{}s formula}: for any function $f\in L^1(G)$,  the periodization map $\Phi(\dot g)=\sum_{\lambda\in \Lambda} f(g+\lambda), \ \dot g\in G/\Lambda$ is well defined almost everywhere in $G/\Lambda$, belongs to $L^1(G/\Lambda),$ and
\begin{align} \label{Weil}
\int_G f(g)dg = |Q_\Lambda| \int_{G/\Lambda} \sum_{\lambda\in\Lambda}  f(g+\lambda) d\dot g.
\end{align}
This formula is a special case of \cite[Theorem 3.4.6]{RS68}. The constant $|Q_\Lambda|$, called the \emph{lattice size}, appears in (\ref{Weil}) because $G/\Lambda$ is equipped with the normalized Haar measure $d\dot g$.

\begin{definition}[Dual integrable representations (\cite{HSWW})]
Let $G$ be an LCA group, and  let $\pi$ be a unitary representation of $G$ on a Hilbert space $\mathcal H$. We say $\pi$ is {\it dual integrable} if there exists a sesquilinear map $[\cdot, \cdot]_\pi: \mathcal H\times \mathcal H\to L^1(\widehat G)$, called {\it bracket map} for $\pi,$ such that
$$
\langle \phi, \pi(g)\psi\rangle_{\mathcal H} = \int_{\widehat G} [\phi, \psi]_\pi(\chi) e_{-g}(\chi) d\chi \quad  \forall \ g\in G \, , \ \ \forall \ \phi, \psi\in \mathcal H .
$$
\end{definition}

\begin{example}\label{translation}
Let $\Lambda$ be a lattice in an LCA group $G$. For any $\lambda\in \Lambda$, define $T_\lambda\phi(g)= \phi(g-\lambda)$ on $\phi\in L^2(G)$ and $M_\lambda h(\chi)= e_{\lambda}(\chi) h(\chi)$ on $h \in L^2(\widehat G)$.
Let us denote with $Q_{\Lambda^\perp}$ a cross section for the annihilator lattice $\Lambda^\perp$. Thus, by Plancherel formula (\ref{Plancherel}) and Weil's formula (\ref{Weil}) we have
\begin{align*}
\langle \phi, T_\lambda\psi\rangle_{L^2(G)} & = \langle \mathcal F_G (\phi) , M_\lambda \mathcal F_G (\psi)\rangle_{L^2(\widehat G)} = \int_{\widehat G} \mathcal F_G (\phi)(\chi) \overline{\mathcal F_G (\psi)(\chi)} e_{-\lambda}(\chi) d\chi \\
& = |Q_{\Lambda^\perp}| \int_{\widehat G/{\Lambda^\perp}} \sum_{\tilde\lambda\in\Lambda^\perp} \mathcal F_G (\phi)(\dot\chi+\tilde \lambda) \overline{\mathcal F_G (\psi)(\dot\chi+\tilde \lambda)} e_{-\lambda}(\dot\chi+\tilde\lambda) d\dot\chi\\
& = |Q_{\Lambda^\perp}| \int_{\widehat G/{\Lambda^\perp}} \sum_{\tilde\lambda\in\Lambda^\perp} \mathcal F_G (\phi)(\dot\chi+\tilde \lambda) \overline{\mathcal F_G (\psi)(\dot\chi+\tilde \lambda)} e_{-\lambda}(\dot\chi) d\dot\chi.
\end{align*}
Since $\mathcal F_G(\phi) \overline{\mathcal F_G (\psi)} \in L^1(\widehat G)$, we have that
$$
[\phi, \psi]_T(\dot\chi): = |Q_{\Lambda^\perp}| \sum_{\tilde \lambda\in \Lambda^\perp} \mathcal F_G (\phi) (\dot\chi+\tilde\lambda)\overline{\mathcal F_G (\psi)(\dot\chi+\tilde\lambda)}\, \qquad a. e. \quad \dot\chi\in \widehat G/{\Lambda^\perp}
$$
defines a sesquilinear map $[\cdot , \cdot]_T : L^2(G) \times L^2(G) \to L^1(\widehat G/\Lambda^\perp)$, so $T$ is a dual integrable representation of $\Lambda$ on $\mathcal H = L^2(G)$.
\end{example}

A relevant application of dual integrable representations is the possibility to characterize bases of unitary orbits in terms of their associated bracket maps. The following result has been proved in \cite[Proposition 5.1]{HSWW}.

\begin{theorem}\label{Result in HSWW}
Let $G$ be a countable abelian group, let $\pi$ be a dual integrable representation of $G$ on a Hilbert space $\mathcal H$, and let $\phi\in \mathcal H$. The system $\{\pi(g)\phi: g\in G\}$ is orthonormal in $\mathcal H$ if and only if $[\phi, \phi]_\pi(\chi)= 1$ for almost every $\chi\in \widehat G$.
\end{theorem}

As a consequence of Theorem \ref{Result in HSWW} and Example \ref{translation}, and of the basic fact $|Q_\Lambda| |Q_{\Lambda^\perp}|=1$ (for completeness, we have provided a proof in the appendix), we have the following result.

\begin{corollary}\label{orthonormal translation system}
Let $T$ and $\Lambda$ be as in Example \ref{translation}, and let $\phi\in L^2(G)$. Then the system of translates $\{T_\lambda \phi: \lambda\in \Lambda\}$ is an orthonormal system in $L^2(G)$ if and only if
$$
\sum_{\tilde\lambda\in \Lambda^\perp} |\mathcal F_G(\phi) (\chi+\tilde\lambda)|^2 = |Q_\Lambda| \quad  a.e. \ \ \chi\in \widehat G\,.
$$
\end{corollary}

\section{Proof of Theorem \ref{main}}

In this section we shall prove Theorem \ref{main} and its corollaries.

\begin{proof}[Proof of  Theorem \ref{main}]\ \\

1)\,$\Rightarrow$\,4) It is well-known (\cite{Rudin-FA}) that, for any cross section $Q_\Lambda$, the exponential set $E_{\Lambda^\perp}$ is an orthogonal basis for $L^2(Q_\Lambda)$. Thus, for all $f \in L^2(Q_\Lambda)$,
\begin{equation}\label{eq:Plancherel}
\sum_{\tilde\lambda \in \Lambda^\perp} |\langle f, \frac{1}{\sqrt{|Q_\Lambda|}}e_{\tilde\lambda}\rangle_{L^2(Q_\Lambda)}|^2 = \|f\|^2_{L^2(Q_\Lambda)}\,.
\end{equation}
Since condition (1) says that $\Omega$ is contained in some cross section $Q_\Lambda$, then
the previous identity still holds for all $f \in L^2(\Omega)$. Hence $E_{\Lambda^\perp}$ is a tight frame for $L^2(\Omega)$ with constant $|Q_\Lambda|$.\\

4)\,$\Rightarrow$\,1) Suppose, by contradiction, that $\Omega$ is not a subtiling set. Then we claim that for all cross section $Q_\Lambda$  there exist $\lambda_1, \lambda_2 \in \Lambda$, $\lambda_2 \neq 0$, such that
\begin{equation}\label{eq:claim}
|(Q_\Lambda + \lambda_1) \cap \Omega \cap (\Omega + \lambda_2)| > 0.
\end{equation}
If this is true, then let $\Omega_1 = (Q_\Lambda + \lambda_1) \cap \Omega \cap (\Omega + \lambda_2)$, and $\Omega_2 = \Omega_1 - \lambda_2$. Both are subsets of $\Omega$ with positive measure and, since $\lambda_2 \neq 0$, they are disjoint because $\Omega_1 \subset Q_\Lambda + \lambda_1$ and $\Omega_2 \subset Q_\lambda + \lambda_1 - \lambda_2$. Therefore, the function
$$
f = {\bf 1}_{\Omega_1} - {\bf 1}_{\Omega_2}
$$
is nonzero and belongs to $L^2(\Omega)$. Then, for all $\tilde\lambda \in \Lambda^\perp$ we have
$$
\langle f, e_{\tilde\lambda}\rangle_{L^2(\Omega)} = \int_{\Omega_1} e_{\tilde\lambda}(g) dg - \int_{\Omega_2} e_{\tilde\lambda}(g) dg = \int_{\Omega_1} \big(e_{\tilde\lambda}(g) - e_{\tilde\lambda}(g - \lambda_2)\big) dg = 0 .
$$

This implies that the system $E_{\Lambda^\perp}$ can not be a frame for $L^2(\Omega)$.

In order to prove (\ref{eq:claim}), let us proceed by contradiction and suppose that for all $\lambda \in \Lambda$ and all $\lambda^* \in \Lambda$, $\lambda^* \neq 0$ we have\vspace{-2ex}
$$
|(Q_\Lambda + \lambda) \cap \Omega \cap (\Omega + \lambda^*)| = 0.\vspace{-2ex}
$$
Now take $\lambda' \in \Lambda$, $\lambda' \neq 0$. By definition of cross section, we have
$$
\Omega \cap (\Omega + \lambda') = \bigsqcup_{\lambda \in \Lambda} (Q_\Lambda + \lambda) \cap \Omega \cap (\Omega + \lambda')
$$
which implies that $|\Omega \cap (\Omega + \lambda')| = 0$. Hence, $\Omega$ would be a subtiling set of $G$ by $\Lambda$, which is a contradiction.

1)\,$\Rightarrow$\,2) Since (\ref{eq:Plancherel}) holds, we can obtain 2) by choosing $f = \overline{e_\chi}\,{\bf 1}_\Omega$.\\

2)\,$\Rightarrow$\,3) This follows as an application of Corollary \ref{orthonormal translation system}.\\

3)\,$\Rightarrow$\,1) By orthogonality, we have that for all $\lambda \in \Lambda$, $\lambda \neq 0$\vspace{-2ex}
$$
0 = \langle {\bf 1}_{\Omega}, {\bf 1}_{\Omega}(\cdot - \lambda)\rangle_{L^2(G)} = |\Omega \cap (\Omega + \lambda)|\vspace{-2ex}
$$
so $\Omega$ is sub-tiling.
\end{proof}

\begin{proof}[Proof of Corollary \ref{cor:Pedersen}]
If $(\Omega,\Lambda)$ is a tiling pair then it is well-known that $E_{\Lambda^\perp}$ is an orthogonal basis for $L^2(\Omega)$. To prove the converse, assume by contradiction that $\Omega$ is not tiling. Then one of the following cases holds
\begin{itemize}
\item[i.] $\Omega$ is a strictly sub-tiling set, i.e. there exists a cross section $Q_\Lambda$ such that $\Omega \subset Q_\Lambda$ and $|Q_\Lambda \setminus \Omega| > 0$.
\item[ii.] $\Omega$ is not a sub-tiling set, so that (\ref{eq:claim}) holds.
\end{itemize}
For case i., observe that the assumption of $E_{\Lambda^\perp}$ being an orthogonal basis for $L^2(\Omega)$ implies
$$
\sum_{\tilde\lambda \in \Lambda^\perp} |\langle f, \frac{1}{\sqrt{|\Omega|}}e_{\tilde\lambda}\rangle_{L^2(\Omega)}|^2 = \|f\|_{L^2(\Omega)}^2 \quad \forall \ f \in L^2(\Omega).
$$
On the other hand, since $E_{\Lambda^\perp}$ is an orthogonal basis for $L^2(Q_\Lambda)$, we have
$$
\sum_{\tilde\lambda \in \Lambda^\perp} |\langle f{\bf 1}_{\Omega}, \frac{1}{\sqrt{|Q_\Lambda|}}e_{\tilde\lambda}\rangle_{L^2(Q_\Lambda)}|^2 = \|f{\bf 1}_\Omega\|_{L^2(Q_\Lambda)}^2 \quad \forall \ f \in L^2(\Omega)
$$
so that $|\Omega| = |Q_\Lambda|$, which contradicts i.

For case ii., in Theorem \ref{main} we already proved that $E_{\Lambda^\perp}$ can not even be a frame.
\end{proof}

\begin{proof}[Proof of Corollary \ref{cor:modulates}]
Assume 4) holds, i.e. that $E_{\Lambda^\perp}$ is a tight frame for $L^2(\Omega)$ with constant $|Q_\Lambda|$. Then
$$
\sum_{\tilde\lambda \in \Lambda^\perp} |\langle f , M_{\tilde\lambda}\psi\rangle_{L^2(\Omega)}|^2 = \sum_{\tilde\lambda \in \Lambda^\perp} |\langle f \overline{\psi}, e_{\tilde\lambda}\rangle_{L^2(\Omega)}|^2 = |Q_{\Lambda}|\, \|f\overline{\psi}\|_{L^2(\Omega)}^2 \quad \forall \ f \in L^2(\Omega) \, .
$$
Since $A\|f\|_{L^2(\Omega)}^2 \leq \|f\overline{\psi}\|_{L^2(\Omega)}^2 \leq B \|f\|_{L^2(\Omega)}^2$, this proves 5). Conversely, assume 5) holds. Then, since $A > 0$, for all $f \in L^2(\Omega)$ we can write
$$
\sum_{\tilde\lambda \in \Lambda^\perp} |\langle f , e_{\tilde\lambda}\rangle_{L^2(\Omega)}|^2 = \sum_{\tilde\lambda \in \Lambda^\perp} |\langle f/\overline{\psi} , M_{\tilde\lambda}\psi\rangle_{L^2(\Omega)}|^2 \, ,
$$
so that, by the hypotheses on $\psi$, we get
$$
\frac{A}{B}|Q_{\Lambda}|\, \|f\|_{L^2(\Omega)}^2 \leq \sum_{\tilde\lambda \in \Lambda^\perp} |\langle f , e_{\tilde\lambda}\rangle_{L^2(\Omega)}|^2 \leq \frac{B}{A}|Q_{\Lambda}|\, \|f\|_{L^2(\Omega)}^2 \quad \forall \ f \in L^2(\Omega).
$$
Thus $E_{\Lambda^\perp}$ is a frame, hence proving 4). Observe that, by Theorem \ref{main}, this implies that $E_{\Lambda^\perp}$ is a tight frame with constant $|Q_\Lambda|$, hence improving the inequalities above.
\end{proof}

\section{Comments on related work} \label{comments}

The statement of Theorem \ref{main} relating sutilings pairs $(\Omega, \Lambda)$ for a lattice $\Lambda$ and frames of exponentials for $L^2(\Omega)$ is new. However, the proof has similarities with existing proofs of the similar statement for tiling sets by lattices and orthonormal bases of exponentials.

The result in Corollary \ref{cor:Pedersen} is proved by B. Fuglede (\cite{Fue74}, Lemma 6) for $G=\mathbb R^n$ and by S. Pedersen (\cite{Pedersen}, Theorem 3.6) for LCA groups.

In both papers, \cite{Fue74} and \cite{Pedersen}, as well as in the present work, the implication 1)\,$\Rightarrow$\,4) is done in the same way by observing that $E_{\Lambda^\perp}$ is an orthogonal basis of exponentials of $L^2(Q_\Lambda)$. 
As for the implication 4)\,$\Rightarrow$\,1) both papers give a direct proof of the fact that $\Omega \cap (\Omega + \lambda)$ have measure zero for all $\lambda \in \Lambda$ and that also the set $\displaystyle G \setminus \bigcup_{\lambda \in \Lambda}(\Omega + \lambda)$ has measure zero. In the present paper we give an argument by contradiction assuming that 1) does not hold and exhibiting a non-zero function in $L^2(\Omega)$ which is perpendicular to all $e_{\tilde \lambda}, \tilde \lambda \in \Lambda^\perp.$

The manuscript \cite{Iosevich-Expository} by A. Iosevich gives a proof of the equivalence of 1) and 2) in Theorem \ref{main} for a tiling set $\Omega \in \mathbb R^n$ by a lattice $\Lambda.$ It is then stated without proof that 2) and 4) are equivalent by a density argument. The proof of 1)\,$\Leftrightarrow$\,2) in \cite{Iosevich-Expository} goes as follows. Consider the functions 
$$
f(x) \equiv \sum_{\lambda \in \Lambda} {\bf 1}_\Omega (x+\lambda)\,, \quad x\in \mathbb R^n \qquad \text{and} \qquad H(\xi)\equiv \sum_{\tilde\lambda\in \Lambda^\perp} |\mathcal F_{\mathbb R^n}({\bf 1}_\Omega) (\xi+\tilde\lambda)|^2\,.
$$
Using Fourier Analysis it can be shown that, as a periodic function in $L^2(Q_{\Lambda^\perp})$, the Fourier coefficients of $H$ are
\begin{equation} \label{Eq:comments1}
\widehat H (\lambda) = \frac{|\Omega \cap (\Omega + \lambda)|}{|Q_{\Lambda^\perp}|}\,, \quad \lambda \in \Lambda\,,
\end{equation}
and the Fourier coefficients of $f$, as a periodic function in $L^2(Q_\Lambda)$, are
\begin{equation} \label{Eq:comments2}
\widehat f (\tilde \lambda) = \frac{\mathcal F_{\mathbb R^n} ({\bf 1}_\Omega)(\tilde \lambda)}{|Q_\Lambda|}\,, \quad \tilde\lambda \in \tilde \Lambda\,.
\end{equation}
Assuming that 1) of Theorem \ref{main} holds for a tiling set $\Omega$, equation (\ref*{Eq:comments1}) shows, using the Fourier inversion theorem, that $H$ is constant with value $|\Omega|\, |Q_\Lambda|$ a.e., since $|Q_\Lambda|\, |Q_{\Lambda^\perp}|=1.$ This shows 2) of Theorem \ref{main}. Conversely, assuming 2) of Theorem \ref{main} holds, equation (\ref*{Eq:comments1}) shows that $\Omega$ is a subtiling set of $\mathbb R^n$ by $\Lambda$. Equation (\ref*{Eq:comments2}) is then used to show that $f(x)=1$ a. e., which shows 1).

A proof along the lines described above can be designed for LCA groups and subtiles. In the proof given in the present work, we have proved the equivalence of 1), 2) and 3) in Theorem \ref{main} by using the notion of bracket map (see \cite{HSWW} and the references therein) and the characterization of frame sequences of translates of a single function along lattices stated in Theorem \ref{Result in HSWW} and Corollary \ref{orthonormal translation system}. As in \cite{Iosevich-Expository} our paper also uses Fourier Analysis to compute Fourier coefficients (see Example \ref{translation}).

\appendix

\section{Appendix}

We provide here a self-contained proof of the following basic result.
\begin{lemma}
Let $G$ be an LCA group, let $\widehat{G}$ be its dual group, and let the Haar measures $dg$ on $G$ and $d\chi$ on $\widehat{G}$ be chosen in such a way that the group Fourier transform be an isometry, i.e. such that (\ref{Plancherel}) holds. Let $\Lambda \subset G$ be a lattice subgroup of $G$, and denote with $Q_\Lambda \subset G$ a cross section for $G/\Lambda$, let $\Lambda^\perp \subset \widehat G$ be the dual lattice of $\Lambda$, and denote with $Q_{\Lambda^\perp} \subset \widehat{G}$ a cross section for $\widehat{G}/\Lambda^\perp$. Then
$$
|Q_\Lambda| |Q_{\Lambda^\perp}| = 1
$$
i.e. the product of the size of the lattice times the size of the dual lattice, computed with respect to Haar measures satisfying (\ref{Plancherel}), is 1.
\end{lemma}

\begin{proof}
Observe first that, by (\ref{Plancherel}), we have
\begin{align}\label{dummy}
|Q_\Lambda| & = \int_G |{\bf 1}_{Q_\Lambda}(g)|^2 dg = \int_{\widehat{G}} |\mathcal{F}_G({\bf 1}_{Q_\Lambda})(\chi)|^2 d\chi \nonumber\\
& = \int_{Q_{\Lambda^\perp}} \sum_{\tilde{\lambda} \in \Lambda^\perp} |\mathcal{F}_G({\bf 1}_{Q_\Lambda})(\chi + \tilde{\lambda})|^2 d\chi
\end{align}
where the last identity is due to the fact that $(Q_{\Lambda^\perp},\Lambda^\perp)$ is a tiling pair for $\widehat{G}$ by definition of cross-section. We can thus apply the same argument used to prove point 2) of Theorem \ref{main}: since $\{e_{\tilde{\lambda}} \, : \, \tilde{\lambda} \in \Lambda^\perp\}$ is an orthogonal basis of $L^2(Q_\Lambda)$, then it satisfies (\ref{eq:Plancherel}), so by choosing $f = \overline{e_\chi}{\bf 1}_{Q_\Lambda}$ we get
$$
\sum_{\tilde{\lambda} \in \Lambda^\perp} |\mathcal{F}_G({\bf 1}_{Q_\Lambda})(\chi + \tilde{\lambda})|^2 = |Q_\Lambda|^2 \quad \textnormal{a.e.} \ \chi \in \widehat{G}.
$$
The claim then follows by inserting this identity in (\ref{dummy}).
\end{proof}

\end{document}